\documentclass[a4paper,12pt,reqno]{amsart}
\usepackage{eurosym}
\usepackage{amsfonts}
\usepackage{amsmath}
\usepackage{amssymb}
\usepackage[a4paper]{geometry}
\usepackage{mathrsfs}
\usepackage{csquotes}
\usepackage{enumitem}
\usepackage{color}
\usepackage[colorlinks]{hyperref}
\usepackage{amsfonts}
\usepackage{epsfig}
\usepackage{epstopdf}
\usepackage{graphicx}
\usepackage{float}
\usepackage{caption}
\usepackage{subfig}

\renewcommand\eqref[1]{(\ref{#1})} 
\numberwithin{equation}{section}
\newtheorem{theorem}{Theorem}
\theoremstyle{plain}

\newtheorem{lemma}{Lemma}

\begin{document}
\title[Inverse problems for subelliptic heat equations]{Inverse problems of
identifying the time-dependent source coefficient for subelliptic heat
equations}
\author[M. Ismailov]{Mansur I. Ismailov}
\address{Mansur I. Ismailov: \endgraf
	Department of Mathematics \endgraf
	Gebze Tehnical University \endgraf
Turkey \endgraf
		and Department of Mathematics \endgraf
				Khazar University \endgraf
Baku AZ1096 \endgraf	
	Azerbaijan \endgraf		
	\textit{E-mail address} \textrm{mismailov@gtu.edu.tr}}
\author[T. Ozawa]{Tohru Ozawa}
\address{Tohru Ozawa: \endgraf
	Department of Applied Physics \endgraf
	Waseda University \endgraf
Tokyo 169-8555 \endgraf
		Japan \endgraf
	\textit{E-mail address} \textrm{txozawa@waseda.jp} }
\author[D. Suragan]{Durvudkhan Suragan}
\address{ Durvudkhan Suragan: \endgraf
	Department of Mathematics \endgraf
Nazarbayev University \endgraf
Astana 010000 \endgraf
Kazakhstan \endgraf
	\textit{E-mail address} \textrm{durvudkhan.suragan@nu.edu.kz} }
\thanks{This work was supported by the NU grant 20122022FD4105. The
second author is supported in part by JSPS Kakenhi 18KK0073, 19H00644. No new data was collected or generated during the course of
the research.}

\begin{abstract}
We discuss inverse problems of determining the time-dependent source
coefficient for a general class of subelliptic heat equations. We show that
a single data at an observation point guarantees the existence of a (smooth)
solution pair for the inverse problem. Moreover, additional data at the
observation point implies an explicit formula for the time-dependent source
coefficient. We also explore an inverse problem with nonlocal additional data,
 which seems a new approach even in the Laplacian case.  
 \end{abstract}

\maketitle

\tableofcontents

\section{Introduction}

Let $X = (X_1, \dots, X_m)$ be a system of real smooth vector fields
defined over a subset $W$ of $\mathbb{R}^{d},\; d\geq 2,$ satisfying the
H\"ormander condition (H): There exists a natural number $r$ such that the
vector fields $X_1, \dots, X_m$ together with their $r$ commutators span the
tangent space at each point of $W$.

Let $\Omega\subset W$ be a bounded connected open subset with smooth
boundary $\partial \Omega$ non-characteristic of $X = (X_1, \dots, X_m)$. In 
$\Omega$, we consider the following inverse problem of finding a pair $(u,p)$
: 
\begin{equation}  \label{inv1intro}
\left\{ 
\begin{split}
\partial_{t} u(x,t) - \Delta_X u(x,t)&=p(t) u(x,t)+f(x,t),\quad \text{in}~
\Omega\times(0,T), \\
u(x,0)&=\varphi (x),\qquad x\in\Omega, \\
u(x,t)&=0,\qquad \text{on}\; \partial\Omega\times(0,T),
\end{split}
\right.
\end{equation}
with an additional condition. Here 
\begin{equation*}
\Delta_X := - \sum_{i=1}^m X_i^\ast X_i,
\end{equation*}
with $X_{i}^\ast = -X_{i}- \text{div} X_{i}$, is a self-adjoint subelliptic
operator (see \cite{ChenChen}).

In the classical case when $X_i=\partial_{x_{i}},\; i=1,\dots, d$, one has the
usual (elliptic) Laplacian $\Delta$ instead of the subelliptic operator $%
\Delta_X$ in \eqref{inv1intro} and there is a vast of literature on such
parabolic inverse problems, especially, in 1D-cases. See, e.g. \cite{HLIK19}
and references therein. Note that, in the classical case, the existence and
uniqueness result of this inverse problem was established by Cannon, Lin and
Wang \cite{CLW91} and \cite{CLW92}, by using the method of retardation of the time
variable and a priori estimates. Without any claim of completeness, we
refer an interested reader to \cite{CTY10} and \cite{KOS20}
for analysis of related inverse problems in the classical setting.



The present paper aims to analyse inverse problems of recovering the
time-dependent source parameter $p(t)$ in the Cauchy-Dirichlet problem for
the subelliptic heat equation \eqref{inv1intro}. First, in order to
find a pair $(u,p)$, we fix a point $q\in\Omega$ as an observation point for
some time-dependent quantity. So, by using this additional date we recover
the time-dependent source parameter $p(t)$.  Interestingly, we discover that this method works well to study an inverse
problem for the same model but with nonlocal additional data. The latter approach seems new even in the Laplacian case. 
Our proofs rely on subelliptic spectral theory arguments.

Moreover, we also state that the solution to the inverse problem can be
found explicitly in the case of an increase in the number of
overdetermined data by potential theory techniques.

We organize our paper as follows. In Section \ref{sec2}, we prove the existence and uniqueness result with a single datum at an observation point by using a spectral theory approach. In Section \ref{sec3}, the time-dependent source parameter is found in an explicit form by applying the potential
theory arguments. In Section \ref{sec4}, our technique from Section \ref{sec2} is
applied to treat a nonlocal case. Some interesting particular models are
discussed in Section \ref{sec5}.

\section{Single datum}

\label{sec2} Let $\Omega $ be a bounded connected open subset with smooth
boundary $\partial \Omega $ non-characteristic of $X=(X_{1},\dots ,X_{m})$
and let $|H|>0$, where the set $H$ is defined in \eqref{condH}. Consider the following inverse problem
of finding a pair $(u,p)$: 
\begin{equation}
\left\{ 
\begin{split}
\partial _{t}u(x,t)& -\Delta _{X}u(x,t)=p(t)u(x,t)+f(x,t),\quad \text{in}%
~\Omega \times (0,T), \\
u(x,0)& =\varphi (x),\qquad x\in \Omega , \\
u(x,t)& =0,\qquad \text{on}\;\partial \Omega \times  (0,T),
\end{split}%
\right.  \label{eq13}
\end{equation}%
where $\varphi \in C^{2k}(\overline{\Omega })$ and $f\in C^{2k,\,0}(%
\overline{\Omega }\times [0,T])$ for integer $k>\frac{\tilde{\nu}}{4}+1$.
Here and after we understand $g\in C^{1}(\Omega )$ if $Xg\in C(\Omega ).$

The operator $\Delta_X$ is well-defined on $\{ u \in H^1_{X,0}(\Omega):
\Delta_X u \in L^2(\Omega) \}.$ Recall that $H^1_X(\Omega) = \{ u \in
L^2(\Omega): X_i u \in L^2(\Omega), 1 \leq i \leq m \}$ and $
H^1_{X,0}(\Omega)$ is the closure of $C_0^\infty(\Omega)$ in $
H^1_{X}(\Omega) $. 

For any $x \in \bar{\Omega}$ and given $1\leq k \leq r$, let $V_{k}(x)$ be
the subspaces of the tangent space at $x$ spanned by all commutators of $
X_{1}, . . . , X_{m}$ with length at most $k$. Note that the number $r$ in the statement (H) in the
introduction is called the H\"ormander index. The Hausdorff dimension $\nu$
of $\Omega$ (or it can be also called the M\'etivier index of $\Omega$) is
defined as

\begin{equation*}
\nu:=\sum_{k=1}^r k\left(\nu_k-\nu_{k-1}\right)
\end{equation*}
with $\nu_0=0$. Here it is assumed that for each $x \in \bar{\Omega}$, $%
\text{dim} V_{k}(x)$ is a constant denoted by $\nu_{k}$ in a neighborhood of 
$x$.

Let us consider the following Dirichlet spectral problem of finding a
nontrivial function $\phi$ and eigenvalue $\lambda $: 
\begin{equation*}
\left\{ 
\begin{split}
-\Delta_X \phi(x)&=\lambda \phi(x),\quad \text{in}~\Omega , \\
\phi(x)& =0,\text{ \ \ on}\ \partial \Omega.
\end{split}
\right.
\end{equation*}

By using the spectral theorem for compact self-adjoint operators, it can be
shown that the eigenspaces are finite-dimensional and that the Dirichlet
eigenvalues $\lambda $ are real, positive, and have no limit point. The eigenspaces are orthogonal in the space of square-integrable functions
and consist of smooth functions. In fact, the system of eigenfunctions $\phi
_{n}(x),$ $n\in \mathbb{N},$ are complete orthonormal system in $
L_{2}(\Omega )$.  

The eigenvalues can be arranged in increasing order:

\begin{equation*}
0<\lambda _{1}\leq \lambda _{2}\leq \cdots \leq \lambda _{n}\rightarrow
\infty,
\end{equation*}
where each eigenvalue is counted according to its geometric multiplicity. It is well-known that 
Weyl's asymptotic formula  for the Dirichlet Laplacian ($\nu=d$) holds 
\begin{equation}
\lambda _{n}\sim n ^{\frac{2}{\nu}}.  \label{eq14}
\end{equation}

The asymptotic formula \eqref{eq14} fails to hold for general H\"ormander
vector fields not satisfying the so-called M\'etivier condition. However, recently, in 
\cite{ChenChen} for the operator $\Delta_X$ it was proved that Weyl's asymptotic formula

\begin{equation}
\lambda _{n}\sim n ^{\frac{2}{\tilde{\nu}}}  \label{eq15}
\end{equation}
holds if and only if $|H|>0$. Here 

\begin{equation}\label{condH}
H:=\{x \in \Omega \mid \nu(x)=\tilde{\nu}
\}\end{equation}
and $\nu(x):=\sum_{j=1}^r j\left(\nu_{j}(x)-\nu_{j-1}(x)\right)$ (with $
\nu_0(x):=0$) is a pointwise homogeneous dimension and 
\begin{equation}
\tilde{\nu}:=\max _{x \in \bar{\Omega}} \nu(x).  \label{Weylfornu}
\end{equation}
Clearly, $\tilde{\nu}=\nu$ if $\text{dim} V_{k}(x)$ is a constant in a
neighborhood of $x$.

Let $q\in \Omega $ be a point such that $\phi _{n}(q)$, $n=1,2,...$, is
bounded with $\phi _{n}(q)\neq 0$ for some $n_{0}\in \mathbb{N}$. Let $%
\mathbb{N}_{q}\subset \mathbb{N}$ be maximal set such that $\phi _{n}(q)\neq 0$
for all $n\in \mathbb{N}_{q}$. We use this point $q\in \Omega $ as an
observation point for the quantity: 
\begin{equation}\label{obspoint}
w(t)=u(q,t),\quad t\in \lbrack 0,T].
\end{equation}
Set $\varphi _{n}=\int_{\Omega }\varphi (x)\phi _{n}(x)dx$ and $%
f_{n}(t)=\int_{\Omega }f(t,x)\phi _{n}(x)dx.$

\begin{lemma}\label{lemma1}
If $\varphi \in C^{2k}(\bar{\Omega})$ and $\Delta _{X}^{m}\varphi =0,$ $%
m=0,...,k-1,$ on $x\in \partial \Omega $ with $k>\frac{\tilde{\nu}}{4}+1$,
then we have%
\begin{equation*}
\sum_{n=1}^{\infty }\lambda _{n}\left\vert \varphi _{n}\right\vert \leq
c^\frac{1}{2}\left\Vert \Delta _{X}^{k}\varphi \right\Vert _{L_{2}}\text{,}
\end{equation*}%
where $c=\sum_{n=1}^{\infty }\frac{1}{\lambda _{n}^{2\left( k-1\right) }}$.
\end{lemma}

\begin{proof}
From Green's identities for the subelliptic operator $\Delta _{X}$ we
get 
\begin{equation*}
\varphi _{n}=(-1)^{k}\frac{1}{\lambda _{n}^{k}}\left( \Delta _{X}^{k}\varphi
\right) _{n},
\end{equation*}%
where $\left( \Delta _{X}^{k}\varphi \right) _{n}=$ $\int_{\Omega }\Delta
_{X}^{k}\varphi (x)\phi _{n}(x)dx$.\ By the Cauchy-Schwartz and Bessel
inequalities we have 
\begin{equation*}
\sum_{n=1}^{\infty }\lambda _{n}\left\vert \varphi _{n}\right\vert
=\sum_{n=1}^{\infty }\frac{1}{\lambda _{n}^{k-1}}\left\vert \left( \Delta
_{X}^{k}\varphi \right) _{n}\right\vert \leq \left(\sum_{n=1}^{\infty }\frac{1}{%
\lambda _{n}^{2\left( k-1\right) }}\right)^\frac{1}{2} \left(\sum_{n=1}^{\infty }\left\vert \left(
\Delta _{X}^{k}\varphi \right) _{n}\right\vert ^{2}\right)^\frac{1}{2}\leq c^\frac{1}{2}\left\Vert \Delta
_{X}^{k}\varphi \right\Vert _{L_{2}}.
\end{equation*}%
The series $c=\sum_{n=1}^{\infty }\frac{1}{\lambda _{n}^{2\left( k-1\right) }%
}$ is convergent by the Weyl-type asymptotic formula \eqref{eq15}, since $k>\frac{%
\tilde{\nu}}{4}+1.$
\end{proof}

Let us introduce:

\begin{equation*}
D_{k}(\Omega )=\left\{ \varphi \in C^{2k}(\bar{\Omega} ): \Delta_X
^{m}\varphi =0,\text{ }m=0,...,k-1\text{ on \ }x\in \partial \Omega\right\}
\end{equation*}
for some $k>\frac{\tilde{\nu}}{4}+1$.

We have the following assumptions about the given functions.

\begin{enumerate}
\item $\varphi \in D_{k}(\Omega )$ with $\varphi _{n}\phi_{n}(q)\geq 0$ for $%
\forall n\in \mathbb{N}_{q}$ and $\varphi _{n_{0}}\phi_{n_{0}}(q)>0$ for
some $n_{0}\in \mathbb{N}_{q};$

\item $f\in C(\Omega \times \lbrack 0,T])$ and $f(x,t)\in D_{k}(\Omega )$
with $f_{n}(t)\phi_{n}(q)\geq 0$ for $\forall t\in \lbrack 0,T]$ and for$\
\forall n\in \mathbb{N}_{q}$;

\item $w\in C[0,T]$ with $w(t)\neq 0$ for $\forall t\in \lbrack 0,T]$ and $%
w(0)=\varphi (q).$
\end{enumerate}

The following theorem is valid for the existence and uniqueness of the
inverse problem.

\begin{theorem}
\label{thm1} 
Assuming conditions (1)-(3) hold, a unique smooth pair $(u,p)$ exists that solves the inverse problem \eqref{eq13} with \eqref{obspoint}.
\end{theorem}

\begin{proof}[Proof of Theorem \protect\ref{thm1}]
The solution of (\ref{eq13}) has the form:

\begin{equation*}
u(x,t)=\sum_{n=1}^{\infty }\left( \varphi _{n}e^{-\lambda
_{n}t+\int_{0}^{t}p(\tau )d\tau }+\int_{0}^{t}f_{n}(\tau )e^{-\lambda
_{n}(t-\tau )+\int_{\tau }^{t}p(s)ds}d\tau \right) \phi_{n}(x)
\end{equation*}
with $\varphi _{n}=\int_{\Omega} \varphi(x)\phi_{n}(x)dx$ and $%
f_{n}(t)=\int_{\Omega} f(t,x)\phi_{n}(x)dx.$ 

Let $r(t)=e^{-\int_{0}^{t}p(%
\tau )d\tau }$. So, we have 
\begin{equation*}
r(t)u(x,t)=\sum_{n=1}^{\infty }\left( \varphi _{n}e^{-\lambda
_{n}t}+\int_{0}^{t}f_{n}(\tau )e^{-\lambda _{n}(t-\tau )}r(\tau )d\tau
\right) \phi _{n}(x).
\end{equation*}%
From the additional condition \eqref{obspoint} it follows that 
\begin{equation*}
r(t)u(q,t)=\sum_{n=1}^{\infty }\left( \varphi _{n}e^{-\lambda
_{n}t}+\int_{0}^{t}f_{n}(\tau )e^{-\lambda _{n}(t-\tau )}r(\tau )d\tau
\right) \phi _{n}(q)=r(t)w(t)
\end{equation*}%
or 
\begin{equation}
r(t)=\frac{\sum_{n=1}^{\infty }\varphi _{n}e^{-\lambda _{n}t}\phi _{n}(q)}{%
w(t)}+\frac{1}{w(t)}\int_{0}^{t}\left( \sum_{n=1}^{\infty }\phi
_{n}(q)f_{n}(\tau )e^{-\lambda _{n}(t-\tau )}\right) r(\tau )d\tau
\label{eq16}
\end{equation}%
in the case $w(t)\neq 0$.

From Green's identities for the subelliptic operator $\Delta_X$ we get 
\begin{equation*}
\varphi _{n}=\frac{(-1)^{k}}{\lambda _{n}^{k}} \int_{\Omega}
\Delta_X^{k}\varphi(x) \phi _{n}(x)dx
\end{equation*}
for $\varphi \in C^{2k}(\bar{\Omega} )$ and $\Delta_X ^{m}\varphi =0,$ $%
m=0,...,k-1,$ on $x\in \partial \Omega $. We have 
\begin{equation*}
\left\vert \varphi _{n}\right\vert =\frac{1}{ \lambda _{n}^{k} }\left\vert
\int_{\Omega} \Delta_X^{k}\varphi(x) \phi _{n}(x)dx\right\vert.
\end{equation*}

Let $\left\vert \phi _{n}(q)\right\vert \leq M$ for some $M=const>0$. The
majorant of the series 
\begin{equation*}
\sum_{n=1}^{\infty }\varphi _{n}e^{-\lambda _{n}t}\phi _{n}(q)
\end{equation*}%
is $\sum_{n=1}^{\infty }\left\vert \varphi _{n}\right\vert,$ that is, by the
Cauchy-Schwartz inequality we have 
\begin{equation*}
\sum_{n=1}^{\infty }\left\vert \varphi _{n}\right\vert =\sum_{n=1}^{\infty }%
\frac{1}{\lambda _{n}^{k}}\left\vert \int_{\Omega }\Delta _{X}^{k}\varphi
(x)\phi _{n}(x)dx\right\vert \leq \left(\sum_{n=1}^{\infty }\frac{1}{\lambda
_{n}^{2k}} \right)^\frac{1}{2} \left(\sum_{n=1}^{\infty }\left\vert \int_{\Omega }\Delta
_{X}^{k}\varphi (x)\phi _{n}(x)dx\right\vert ^{2}\right)^\frac{1}{2}.
\end{equation*}%
Now we apply the Bessel inequality 
\begin{equation*}
\sum_{n=1}^{\infty }\frac{1}{\lambda _{n}^{2k}}\sum_{n=1}^{\infty
}\left\vert \int_{\Omega }\Delta _{X}^{k}\varphi (x)\phi
_{n}(x)dx\right\vert ^{2}\leq \sum_{n=1}^{\infty }\frac{1}{\lambda _{n}^{2k}}%
\left\Vert \Delta _{X}^{k}\varphi \right\Vert _{L_{2}}^{2}.
\end{equation*}
This series is convergent by the Weyl-type asymptotic formula \eqref{eq15}.

The same inequality is true for the second series in \eqref{eq16}. Then the
second kind Volterra equation has a unique continuous solution in $[0,T]$
and it must be positive. Then 
\begin{equation*}
r(t)=e^{-\int_{0}^{t}p(\tau )d\tau }\Longrightarrow p(t)=-\frac{r^{\prime
}(t)}{r(t)}.
\end{equation*}%
The function $r(t)$ is continuously differentiable if the series $%
\sum_{n=1}^{\infty }\varphi _{n}\lambda _{n}e^{-\lambda _{n}t}\phi _{n}(q)$
is uniformly convergent. It is the case since the majorant series $%
M\sum_{n=1}^{\infty }\lambda _{n}\left\vert \varphi _{n}\right\vert $ is
convergent for $k>\frac{\tilde{\nu}}{4}+1$ by Lemma \ref{lemma1}.
\end{proof}

\section{Double data}

\label{sec3} Let $\Omega \subset \mathbb{R}^{d}$ be bounded open set with
piecewise smooth boundary $\partial \Omega $. Consider the inverse problem
of identifying a pair $(u,p)$: 
\begin{equation}
\left\{ 
\begin{split}
\partial _{t}u(x,t)& -\Delta _{X}u(x,t)=p(t)u(x,t)+f(x,t),\quad \text{in}%
~\Omega \times (0,T), \\
u(x,0)& =\varphi (x),\qquad x\in \Omega , \\
u(x,t)& =0,\qquad \text{on}\;\partial \Omega \times (0,T),
\end{split}%
\right.  \label{Sec3eq1}
\end{equation}%
where $u_{0}\in C_{0}^{2}(\Omega )$ and $f\in C_{0}^{2,\,0}(\Omega \times
[0,T])$ are given functions.
Assume that there exists $q\in \Omega $ such that $f(q,t)\in C^{1}[0,T]$ is
continuously differentiable function with $f(q,t)\neq 0$. We now fix $q\in
\Omega $ as an observation point for two time-dependent quantities: 
\begin{equation}
w_{1}(t):=v_{1}(t,q),\quad w_{2}(t):=v_{2}(t,q),\quad t\in [0,T].
\label{eq2}
\end{equation}
Here 
\begin{equation}
\left\{ 
\begin{split}
\partial _{t}v_{1}(x,t)& -\Delta_X v_{1}(x,t)=p(t)v_{1}(x,t)+f(x,t),\quad 
\text{in}~\Omega \times (0,T), \\
v_{1}(x,0)& =\varphi  (x),\qquad x\in \Omega , \\
v_{1}(x,t)& =0,\qquad \text{on}\;\partial \Omega \times (0,T),
\end{split}
\right.  \label{eq3}
\end{equation}
and 
\begin{equation}
\left\{ 
\begin{split}
\partial _{t}v_{2}(x,t)& -\Delta_X v_{2}(x,t)=p(t)v_{2}(x,t)+\Delta_X
f(x,t),\quad \text{in}~\Omega \times (0,T), \\
v_{2}(x,0)& =\Delta_X \varphi  (x),\qquad x\in \Omega , \\
v_{2}(x,t)& =0,\qquad \text{on}\;\partial \Omega \times (0,T).
\end{split}
\right.  \label{eq4}
\end{equation}

\begin{theorem}
\label{thmmain} Let $\varphi  \in C_{0}^{2}(\Omega )$ and $f\in
C_{0}^{2,\,0}(\Omega \times [0,T])$. Let $q\in\Omega$ be such that $%
f(q,t)\in C^{1}[0,T]$ is continuously differentiable function with $%
f(q,t)\not\equiv0$. Let $w_{1}$ and $w_{2}$ defined in \eqref{eq2} be the
observation data. Then there exists a unique smooth solution pair $(u,p)$
for the inverse problem \eqref{Sec3eq1} with 
\begin{equation*}
p(t)w_{1}(t)=w^{\prime}_{1}(t)-w_{2}(t)-f(q,t).
\end{equation*}
\end{theorem}

Note that if $p$ is found, then the solution $u$ of the direct problem has the following representation 
\begin{multline*}
u(x,t)=\exp \left(\int_{0}^{t} p(\tau) d \tau\right)\int_{0}^{t}
\int_{\Omega} h_{D}(x,y,t-\tau) \exp \left(-\int_{0}^{\tau} p(s) d s\right)
f(y,\tau)dyd\tau \\
+ \exp \left(\int_{0}^{t} p(\tau) d \tau\right)\int_{\Omega} h_{D}(x,y,t)
\varphi (y)dy,
\end{multline*}
where $h_{D}$ is the subelliptic heat kernel \cite[Theorem 1.1]{ChenChen} for the Cauchy-Dirichlet problem for the
subelliptic heat equation in the cylindrical domain $\Omega\times[0,T)$.

\begin{proof}[Proof of Theorem \protect\ref{thmmain}]
Let us recall the following transformation from \cite{CLW91}: 
\begin{equation*}
v(x,t)=r(t)u(x,t),\quad r(t)=\exp \left( -\int_{0}^{t}p(\tau )d\tau \right) ,
\end{equation*}
that is, 
\begin{equation*}
p(t)=-\frac{r^{\prime }(t)}{r(t)},\quad u(x,t)=\frac{v(x,t)}{r(t)}.
\end{equation*}
Thus, we have 
\begin{equation}
\left\{ 
\begin{split}
\partial _{t}v(x,t)& -\Delta_X v(x,t)=r(t)f(x,t),\quad \text{in}~\Omega
\times (0,T), \\
u(x,0)& =\varphi  (x),\qquad x\in \Omega , \\
u(x,t)& =0,\qquad \text{on}\;\partial \Omega \times ( 0,T).
\end{split}
\right.  \label{eq5}
\end{equation}
We now fix a point $q\in \Omega $ as an observation point for two time
dependent quantities: 
\begin{equation*}
\tilde{w}_{1}(t):=w_{1}(t,q),\quad \tilde{w}_{2}(t):=w_{2}(t,q),\quad t\in
[0,T].
\end{equation*}
Here 
\begin{equation}
\left\{ 
\begin{split}
\partial _{t}w_{1}(x,t)& -\Delta_X w_{1}(x,t)=r(t)f(x,t),\quad \text{in}%
~\Omega \times (0,T), \\
w_{1}(x,0)& =\varphi  (x),\qquad x\in \Omega , \\
w_{1}(x,t)& =0,\qquad \text{on}\;\partial \Omega \times ( 0,T).
\end{split}
\right.  \label{eq6}
\end{equation}
and 
\begin{equation}
\left\{ 
\begin{split}
\partial _{t}w_{2}(x,t)& -\Delta_X w_{2}(x,t)=r(t)\Delta_X f(x,t),\quad 
\text{in}~\Omega \times (0,T), \\
w_{2}(x,0)& =\Delta_X \varphi  (x),\qquad x\in \Omega , \\
w_{2}(x,t)& =0,\qquad \text{on}\;\partial \Omega \times ( 0,T).
\end{split}
\right.  \label{eq7}
\end{equation}

It is known that the solutions of Cauchy-Dirichlet problems (\ref{eq6}) and (%
\ref{eq7}), correspondingly, can be presented by

\begin{equation}
w_{1}(x,t)=\int_{0}^{t}\int_{\Omega} h_{D}(x,y,t-\tau) r(\tau)f(y,\tau) \,
dy+\int_{\Omega} h_{D}(x,y,t)\varphi (y) \, dy  \label{eq8}
\end{equation}
and 
\begin{equation}
w_{2}(x,t)=\int_{0}^{t}\int_{\Omega} h_{D}(x,y,t-\tau) r(\tau)\Delta_X
f(y,\tau) \, dy+\int_{\Omega} h_{D}(x,y,t)\Delta_X \varphi (y) \, dy.
\label{eq9}
\end{equation}
 That is, for $q\in\Omega$ and for all $t \in [0,T] $
we have 
\begin{eqnarray*}
\begin{aligned} &w_{1}(q,t) = \int_0^t \int_{\Omega} h_{D}(q,y,t-\tau)
r(\tau)f(y,\tau) \, dy \, d\tau + \int_{\Omega} h_{D}(q,y,t) \varphi (y) \,
dy=\tilde{w}_{1}(t);\\ &w_{2}(q,t)=\int_0^t \int_{\Omega} h_{D}(q,y,t-\tau)
r(\tau) \Delta_X f(y,\tau) \, dy \, d\tau + \int_{\Omega} h_{D}(q,y,t)
\Delta_X \varphi (y) \, dy=\tilde{w}_{2}(t). \end{aligned}
\end{eqnarray*}
By differentiating $\tilde{w}_{1}$ and applying the Leibniz integral rule, then using $\partial_{t}
h_{D}(x,y,t-\tau)=\Delta_X h_{D}(x,y,t-\tau),\ t>\tau,$ $\partial_{t}
h_{D}(x,y,t)=\Delta_X h_{D}(x,y,t),\;t>0,$ and Green's second identity for the operator $\Delta_X$, we
get 
\begin{eqnarray}  \label{eq10}
\begin{aligned} &\tilde{w}_{1}^{\prime}(t)= \int_0^t \int_{\Omega}
\partial_{t} h_{D}(q,y,t-\tau)r(\tau) f(y,\tau) \, dy \, d\tau +
r(t)f(q,t)\\ &+\int_{\Omega}\partial_{t} h_{D}(q,y,t) \varphi (y) \, dy=\int_0^t
\int_{\Omega} \Delta_X h_{D}(q,y,t-\tau) r(\tau)f(y,\tau) \, dy \, d\tau\\
&+r(t)f(q,t)+\int_{\Omega} \Delta_X h_{D}(q,y,t) \varphi (y) \, dy \\ &=\int_0^t
\int_{\Omega} h_{D}(q,y,t-\tau)r(\tau) \Delta_X f(y,\tau) \, dy \,
d\tau+r(t)f(q,t)\\ &+\int_{\Omega} h_{D}(q,y,t) \Delta_X
\varphi (y)dy=\tilde{w}_{2}(t)+r(t)f(q,t). \end{aligned}
\end{eqnarray}
Thus, we arrive at 
\begin{equation}  \label{eq11}
\tilde{w}_{2}(t)=\tilde{w}_{1}^{\prime}(t)-r(t)f(q,t).
\end{equation}
Moreover, we have the following relations 
\begin{equation*}
\tilde{w}_1(t):=r(t)w_{1}(t),\quad \tilde{w}_2(t):=r(t)w_{2}(t), \quad t\in
[0,T].
\end{equation*}
Plugging in (\ref{eq11}) we obtain 
\begin{equation*}
p(t)w_{1}(t)=w^{\prime}_{1}(t)-w_{2}(t)-f(q,t).
\end{equation*}

Also, the direct Cauchy-Dirichlet
problems \eqref{eq5}, \eqref{eq6}, and \eqref{eq7} have a unique solution.
It implies that there exists a unique classical pair $(v,r)$, that is, $%
(u,p) $ for the inverse problem \eqref{Sec3eq1}. The proof is complete.
\end{proof}

\section{Nonlocal data}

\label{sec4}

Consider the following problem: 
\begin{equation}
\left\{ 
\begin{split}
\partial _{t}u(x,t)& =\Delta _{X}u(x,t)+r(t)f(x,t),\quad \text{in}~\Omega
\times (0,T), \\
u(x,0)& =\varphi (x),\qquad x\in \Omega , \\
u(x,t)& =0,\qquad \text{on}\;\partial \Omega \times ( 0,T),
\end{split}%
\right.  \label{4.1}
\end{equation}%
where $\varphi \in C^{2k}(\Omega )$ and $f\in C^{2k,\,0}(\Omega \times
[0,T]) $ for some integer $k>\frac{\tilde{\nu}}{4}+1$.

Note that the equation \eqref{4.1} is equivalent to \eqref{eq13} (see the proof of Theorem \protect\ref{thmmain}).
Now we consider the inverse problem of finding $(u,r)$ from \eqref{4.1} with
nonlocal datum 
\begin{equation}
\int_{\Omega }\omega (x)u(x,t)dx=w(t),\quad t\in \lbrack 0,T],  \label{4.2}
\end{equation}%
where $\omega (x)\in L_{2}(\Omega ).$

One could also discuss the case when \eqref{4.2} is replaced by the measurement in
Section \ref{sec2}:

\begin{equation*}
u(q,t)=w(t),\quad t\in \lbrack 0,T],  \label{4.3}
\end{equation*}%
where $q$ is a given space observation point in $\Omega $ such that the sequence $\varphi
_{n}(q)$, $n=1,2,\ldots,$ is bounded. \bigskip This is retrieved when one considers the Dirac delta distribution $
\delta (x-q)$ centred at $q$ in equation \eqref{4.2}.

\bigskip The following theorem is valid for the existence and uniqueness of
the inverse problem.

\begin{theorem}
Let the following conditions hold:

\begin{enumerate}
\item $\varphi \in D_{k}(\Omega );$
\item $f\in C(\Omega \times \lbrack 0,T]),$ $f(x,t)\in D_{k}(\Omega )$ and $%
\int_{\Omega }\omega (x)f(x,t)dx\neq 0$ for $\forall t\in \lbrack 0,T];$
\item $w\in C^{1}[0,T]$ with $w(t)\neq 0$ for $\forall t\in \lbrack 0,T]$
and $w(0)=\int_{\Omega }\omega (x)\varphi (x)dx.$
\end{enumerate}
Then there exists a unique smooth pair $(u,r)$ for the inverse problem
\eqref{4.1}-\eqref{4.2}.
\end{theorem}

\begin{proof}
The solution of \eqref{4.1} has the form:

\begin{equation*}
u(x,t)=\sum_{n=1}^{\infty }\left( \varphi _{n}e^{-\lambda
_{n}t}+\int_{0}^{t}f_{n}(\tau )e^{-\lambda _{n}(t-\tau )}r(\tau )d\tau
\right) \phi _{n}(x),
\end{equation*}%
where $\varphi _{n}=\int_{\Omega }\varphi (x)\phi _{n}(x)dx$ and $%
f_{n}(t)=\int_{\Omega }f(x,t)\phi _{n}(x)dx.$ So, we have 
\begin{equation*}
u_{t}(x,t)=\sum_{n=1}^{\infty }\left( -\lambda _{n}\varphi _{n}e^{-\lambda
_{n}t}-\lambda _{n}\int_{0}^{t}f_{n}(\tau )e^{-\lambda _{n}(t-\tau )}r(\tau
)d\tau +f_{n}(t)r(t)\right) \phi _{n}(x).
\end{equation*}%
From the over-determination condition \eqref{4.2} it follows that 
\begin{equation*}
\int_{\Omega }\omega (x)u_{t}(x,t)dx=\sum_{n=1}^{\infty }\left( -\lambda
_{n}\varphi _{n}e^{-\lambda _{n}t}-\lambda _{n}\int_{0}^{t}f_{n}(\tau
)e^{-\lambda _{n}(t-\tau )}r(\tau )d\tau \right) \int_{\Omega }\omega
(x)\phi _{n}(x)dx
\end{equation*}%
\begin{equation*}
+r(t)\int_{\Omega }\omega (x)\sum_{n=1}^{\infty }f_{n}(t)\phi
_{n}(x)dx=\omega^{\prime }(t)
\end{equation*}%
or 
\begin{eqnarray} \label{4.4}
r(t) &=&\frac{w^{\prime }(t)+\sum_{n=1}^{\infty }\lambda _{n}\varphi
_{n}e^{-\lambda _{n}t}\int_{\Omega }\omega (x)\phi _{n}(x)dx}{\int_{\Omega
}\omega (x)f(x,t)dx}  \notag \\
&&   \\
&&+\frac{1}{\int_{\Omega }\omega (x)f(x,t)dx}\int_{0}^{t}\left(
\sum_{n=1}^{\infty }\lambda _{n}f_{n}(\tau )e^{-\lambda _{n}(t-\tau )}d\tau
\int_{\Omega }\omega (x)\phi _{n}(x)dx\right) r(\tau )d\tau  \notag
\end{eqnarray}%
since $\int_{\Omega }\omega (x)\sum_{n=1}^{\infty }f_{n}(t)\phi
_{n}(x)dx=\int_{\Omega }\omega (x)f(x,t)dx.$

Let $\int_{\Omega }\omega ^{2}(x)dx\leq M$ in \eqref{4.2} for some $M=const>0$. The series $$%
\sum_{n=1}^{\infty }\lambda _{n}\varphi _{n}e^{-\lambda _{n}t}\int_{\Omega
}\omega (x)\phi _{n}(x)dx$$ is uniformly convergent since the
majorant series $M\sum_{n=1}^{\infty }\lambda _{n}\left\vert \varphi
_{n}\right\vert $ is convergent by Lemma \ref{lemma1}. Thus, the Volterra integral equation
\eqref{4.4} has a unique continuous solution.
\end{proof}

\section{Particular cases}
\label{sec5}

\subsection{Laplacian}
In the classical case, when $X_i=\partial_{x_{i}},\; i=1,\dots, d$, one has the
usual (elliptic) Laplacian $\Delta$ instead of the subelliptic operator 
$\Delta_X$ in \eqref{inv1intro}. Our approach seems new even in this classical case. 

\subsection{Sub-Laplacians}

Let $\mathbb{G} = (\mathbb{R}^d, \circ)$ be a stratified Lie group and $\{X_1, \dots, X_m\},$ $m\leq d,$ be a system of generators for $\mathbb{G}$, i.e. 
a basis for the first strata of $\mathbb{G}$. It satisfies the
H\"ormander condition (H).  The sum of squares operator $$\Delta_\mathbb{G} = - \sum_{i=1}^m X_i^2$$  is called the sub-Laplacian on $\mathbb{G}$.
$\Delta_\mathbb{G}$ is elliptic if and only if  $\mathbb{G}$ 
is same as $(\mathbb{R}^d, +)$. It is clear that the above results are valid on stratified Lie groups.
Note that the class of stratified Lie groups includes the groups of Iwasawa type and the 
H-type groups. A simple example in  $\mathbb{R}^{3}$ is the operator
$$(\partial_{y}+2x\partial_{z})^{2}+(\partial_{x}-2y\partial_{z})^{2}.$$

\subsection{Baouendi-Grushin operator}

Let $z:=(x,y):=(x_{1},...,x_{m}, y_{1},...,y_{k})\in \mathbb{R}^{m}\times\mathbb{R}^{k}$ with $m,k\geq1$ and $m+k=d$. Let us consider the vector fields
$$X_{i}=\frac{\partial}{\partial x_{i}}, \;i=1,...,m, \;\;\; Y_{j}=|x|^{\gamma}\frac{\partial}{\partial y_{j}},\;\gamma\geq0,\;j=1,...,k.$$
The Baouendi-Grushin operator on $\mathbb{R}^{m+k}$ is defined by
\begin{equation}\label{Grush_op}
\Delta_{\gamma}=\sum_{i=1}^{m}X_{i}^{2}+\sum_{j=1}^{k}Y_{j}^{2}=\triangle_{x}+|x|^{2\gamma}\triangle_{y},
\end{equation}
where $\triangle_{x}$ and $\triangle_{y}$ stand for the standard Laplacians in the variables $x\in \mathbb{R}^{m}$ and $y\in \mathbb{R}^{k}$, respectively. Recall that the Baouendi-Grushin operator for a positive even integer $\gamma$ is a sum of squares of vector fields satisfying the H\"{o}rmander condition (H). In that case, the results of the present paper hold for the Baouendi-Grushin operator (cf. \cite{BCY14}). A simple example in  $\mathbb{R}^{2}$ is the following operator
$$\partial^{2}_{xx}+x^{2}\partial^{2}_{yy}.$$

\end{document}